\documentclass[11pt,draft]{amsart}
\usepackage{mathrsfs}
\usepackage{amssymb}
\usepackage{amsmath}
\usepackage{amsfonts}
\usepackage{amsfonts,enumerate}
\usepackage{pifont}
\usepackage{enumerate}
\usepackage{graphics}
\usepackage{verbatim}
\usepackage{amsthm}

\numberwithin{equation}{section}
\newtheorem{theorem}{Theorem}[section]
\newtheorem{corollary}{Corollary}[section]
\newtheorem{definition}{Definition}[section]
\newtheorem{lemma}{Lemma}[section]
\newtheorem{proposition}{Proposition}[section]
\newtheorem{remark}{Remark}[section]
\newtheorem{example}{Example}[section]

\newcommand{\A}{{\mathcal A}}

\newcommand{\M}{{\mathcal M}}
\newcommand{\N}{{\mathcal N}}

\newcommand{\8}{\infty}
\newcommand{\el}{\ell}

\newcommand{\be}{\begin{eqnarray*}}
\newcommand{\ee}{\end{eqnarray*}}
\newcommand{\beq}{\begin{equation}}
\newcommand{\eeq}{\end{equation}}
\newcommand{\beqn}{\begin{equation*}}
\newcommand{\eeqn}{\end{equation*}}
\newcommand{\bs}{\begin{split}}
\newcommand{\es}{\end{split}}

\numberwithin{equation}{section}

\begin{document}

\title{Noncommutative weak Orlicz spaces and martingale inequalities}


\author{Turdebek N. Bekjan}

\address{College of Mathematics and Systems Science, Xinjiang
University, \newblock Urumqi 830046, China}

\thanks{T.B is partially supported by NSFC grant No.11071204.}

\author{Zeqian Chen}

\address{Wuhan Institute of Physics and Mathematics, Chinese
Academy of Sciences, West District 30, Xiao-Hong-Shan, Wuhan 430071, China}


\author{Peide Liu}

\address{School of Mathematics and Statistics, Wuhan University, Wuhan 430072, China}


\author{Yong Jiao}

\address {Institute of Probability and Statistics , Central South University, Changsha 410075, China}

\date{}
\maketitle

\markboth{T.N.Bekjan, et al}%
{Noncommutative weak Orlicz spaces}

\begin{abstract}
This paper is devoted to the study of noncommutative weak Orlicz spaces and martingale inequalities. Marcinkiewicz interpolation theorem is extended to include noncommutative weak Orlicz spaces as interpolation classes. In particular, we prove the weak type $\Phi$-moment Burkholder-Gundy inequality for noncommutative martingales through establishing a weak type $\Phi$-moment noncommutative Khintchine's inequality for Rademacher's random variables.
\end{abstract}

\section{Introduction}\label{Intro}

Recently, the first two named authors proved an $\Phi$-moment Burkholder-Gundy inequality for noncommutative martingales in \cite{BC}, i.e., the noncommutative analogue of the following inequality \cite{BDG}: Let $\Phi$ be an Orlicz function
with $1<p_{\Phi}\leq q_{\Phi}<\infty.$ If $f=(f_{n})_{n\ge1}$ is a $L_{\Phi}$-bounded martingale, then
\begin{equation}\label{eq:PhiBG}
\int_{\Omega}\Phi \Big [ \Big ( \sum_{n = 1}^{\infty}|d f_n |^{2} \Big )^{\frac{1}{2}} \Big ] d P \approx
\sup_{n\ge1}\int_{\Omega}\Phi(|f_{n}|)dP,
\end{equation}
where $d f = ( d f_n )_{n \geq 1}$ is the martingale difference of $f$ and $``\approx"$ depends only on $\Phi.$ Notice that \eqref{eq:PhiBG} is the well-known Burkholder-Gundy inequality for convex powers $\Phi (t) = t^p$ (see \cite{BG}). In their remarkable paper \cite{PX1997}, Pisier and Xu proved the noncommutative analogue of the Burkholder-Gundy inequality, which triggered a systematic research of noncommutative martingale inequalities. We refer to a recent book by Xu \cite{Xu2008} for an up-to-date exposition of theory of noncommutative martingales. Evidently, the noncommutative $\Phi$-moment Burkholder-Gundy inequality implies those for $L_{\Phi}$ norms, which were already known as particular cases of more general ones established by the first named author in \cite{B2006}.

In this paper, we continue this line of investigation. We will introduce noncommutative weak Orlicz spaces and prove the associated martingale inequalities. In particular, we will prove that noncommutative weak Orlicz spaces can be renormed as Banach spaces under a mild condition of $\Phi,$ and a weak type version of the $\Phi$-moment inequalities for noncommutative martingles obtained recently by the first two named authors \cite{BC}. To the best of our knowledge, this kind of weak type $\Phi$-moment inequalities is new even in the commutative setting.

In \cite{LHW}, the authors prove the Burkholder-Gundy inequality for weak Orlicz spaces, using the arguments of stopping times and  good-$\lambda$ inequalities developed by Burkholder {\it et al} \cite{Burk}. However, the concepts of stopping times and  good-$\lambda$ inequalities are, up to now, not well defined in the generic noncommutative setting (there are some works on this topic, see \cite{AC} and references therein). Instead, interpolation and noncommutative Khinchine inequalities play crucial roles in the proof of the noncommutative Burkholder-Gundy inequality mentioned above. Then, in order to prove the weak type $\Phi$-moment Burkholder-Gundy inequality in the noncommutative setting, we need to prove the associated Khinchine type inequality. There are extensive works on various generalizations of the noncommutative Khinchine inequality in $L_p$-setting \cite{LP1986, LPP}, for instance, see \cite{P2009} and references therein. Unfortunately, our weak type $\Phi$-moment Khinchine inequality can not be obtained directly from ones established previously. By adapting natural and classical techniques in \cite{LP1986, LPP, LPX, MS}, we obtain the required one. This is the key point of this paper.

The remainder of this paper is organized as follows. In Section \ref{Prelimi} we present some preliminaries and notation on the noncommutative weak $L_p$ and Orlicz spaces. Noncommutative weak Orlicz spaces are presented in Section \ref{NcWeakOrliczSpa}. In Section \ref{Interp} we establish a
Marcinkiewicz-type interpolation theorem for noncommutative weak Orlicz spaces and prove that noncommutative weak Orlicz spaces
can be renormed as Banach spaces when $\Phi$ satisfies a mild condition. Finally, in Section \ref{MartingaleInequa}, we will prove the weak type $\Phi$-moment Burkholder-Gundy inequality for noncommutative martingales through establishing a weak type $\Phi$-moment Khintchine's inequality for Rademacher's random variables. The style of proof follows mainly the arguments in \cite{BC}.

In what follows, $C$ always denotes a constant, which may be different in different places. For two nonnegative (possibly infinite) quantities $X$ and $Y,$ by $X \lesssim Y$ we mean that there exists a constant $C>0$ such that $ X \leq C Y,$ and by $X \thickapprox Y$ that $X \lesssim Y$ and $Y \lesssim X.$

\section{Preliminaries}\label{Prelimi}

\subsection{Noncommutative weak $L_p$ spaces}\label{NcWeakLpSpa}

We use standard notation and notions from theory of noncommutative $L_{p}$-spaces. Our main references are \cite{PX2003}
and \cite{Xu2008} (see also \cite{PX2003} for more  historical references). Let $\M$ be a semifinite von Neumann algebra acting on a Hilbert space
$\mathbb{H}$ with a normal semifinite faithful trace $\tau.$ For $0<p<\infty$ let $L_{p}(\M)$ denote the noncommutative $L_p$ space with respect to $(\M, \tau).$
As usual, we set $L_{\infty}(\M,\tau)=\M$ equipped with the operator norm. Also, let $L_{0}(\mathcal{M})$ denote the topological $*$-algebra of
measurable operators with respect to $(\M, \tau).$

For $x\in L_{0}(\M)$ we define
\be
\lambda_{s}(x)=\tau(e^{\perp}_s (|x|))\;(s>0)\; \; \text{and}\; \; \mu_t (x) = \inf \{ s>0:\;
\lambda_s (x) \le t \}\; (t >0),
\ee
where $e_s^{\perp} (|x|) = e_{(s,\infty)}(|x|)$ is the spectral projection of $|x|$ associated with the interval $(s,\infty).$ The function $s\mapsto\lambda_{s}(x)$ is called the distribution function of $x$ and $\mu_{t}(x)$ the generalized singular number of $x.$ We will denote simply by $\lambda (x)$ and  $\mu(x)$ the two functions $s \mapsto \lambda_s (x)$ and $t \mapsto \mu_t (x),$ respectively. It is easy to check that both functions $\lambda (x)$ and $\mu(x)$ are decreasing and continuous from the right on $(0,\infty).$ For further information we refer the reader to \cite{FK}.

For $0<p<\infty,$ we have the following Kolmogorov inequality
\beq\label{eq:kol}
\lambda_{s}(x)\leq\frac{\|x\|_{p}^{p}}{s^{p}},\quad \forall s >0,
\eeq
for any $x\in L_{p}(\M).$ If $x,y$ in $L_0(\M)$, then
\beq\label{eq:tri}
\lambda_{2s}(x+y)\leq \lambda_{s/2}(x)+\lambda_{s/2}(y),\quad \forall s >0.
\eeq
We will frequently use these two inequalities in the sequel.


For $0< p < \8,$ the noncommutative weak $L_p$ space $L^w_p(\mathcal{M})$ is defined as the space
of all measurable operator $x$ such that
\be\begin{split}
\|x\|_{L^w_p}: = \sup_{t > 0} t^{\frac{1}{p}}\mu_{t}(x) < \8.
\end{split}\ee
Equipped with $\|.\|_{L^w_p},$ $L^w_p(\mathcal{M})$ is a quasi-Banach space. However, for $p>1$
$L^w_p(\mathcal{M})$ can be renormed as a Banach space by
\be
x \mapsto \sup_{t >0} t^{-1+ \frac{1}{p}} \int_{0}^{t}\mu_{s}(x)d s.
\ee
On the other hand, the quasi-norm admits the following useful description
\beq\label{eq:WeakLpNormMu}
\|x\|_{L^w_p} = \inf \big \{c>0:\; t ( \mu_t (x)/c )^p \le 1,\; \forall t >0 \big \}.
\eeq
Also, we have a description in terms of distribution function as following
\beq\label{eq:WeakLpNormLmbda}
\|x\|_{L^w_p} = \sup_{s > 0} s \lambda_s (x)^{\frac{1}{p}}.
\eeq

Recall that noncommutative weak $L_p$ spaces can be presented through noncommutative Lorenz spaces, for details see Dodds {\it et al} \cite{DDP1989} and Xu \cite{Xu1991}.

\subsection{Noncommutative Orlicz spaces}\label{NcOrliczSpa}

Recall that noncommutative Orlicz spaces were respectively defined by Kunze \cite{Kunze1990} in an algebraic way
(see also \cite{ARZ} for more general cases) and by Dodds {\it et al} \cite{DDP1989} and by Xu \cite{Xu1991} employing
 Banach space theory. The second approach based on the concept of Banach function spaces, among other properties,
readily indicates similarities with the classical origins. We will take the second approach.

Let $\Phi$  be an Orlicz function on $[0,\infty),$ i.e., a continuous increasing and convex function satisfying
$\Phi(0)=0$ and $\lim_{t\rightarrow
\infty}\Phi(t)=\infty.$ Recall that $\Phi$ is said to satisfy the
$\triangle_2$-condition if there is a constant $C$ such that $\Phi(2t)\leq C\Phi(t)$ for all $t>0.$ In this case,
we denote by $\Phi \in \Delta_2.$
It is easy to check that $\Phi \in \triangle_2$ if and only if for any $a > 0$ there is a constant $C_a>0$ such that
 $\Phi(a t)\leq C_a \Phi(t)$ for all $t>0.$

We will work with some standard indices associated to an Orlicz function. Given an Orlicz function $\Phi.$
Since $\Phi$ is convex, $\Phi' (t)$ is defined for each $t > 0$ except for a countable set of points in which we take
$\Phi'(t)$ as the derivative from the right. Then, we define
\be
a_{\Phi} = \inf_{t>0} \frac{t \Phi'(t)}{\Phi (t)}\quad \text{and} \quad b_{\Phi} = \sup_{t>0} \frac{t \Phi'(t)}{\Phi (t)}.
\ee
\begin{enumerate}[{\rm (1)}]

\item $1 \le a_{\Phi} \le  b_{\Phi} \le \8.$

\item The following characterizations of $a_{\Phi}$ and $b_{\Phi}$ hold:
\be\begin{split}
a_{\Phi} & = \sup \Big \{ p >0:\; t^{-p} \Phi (t)\; \text{is non-decreasing for all}\; t >0 \Big \};\\
b_{\Phi} & = \inf \Big \{ q >0:\; t^{-q} \Phi (t)\; \text{is non-increasing for all}\; t >0 \Big \}.
\end{split}\ee


\item $\Phi \in \triangle_2$  if and only if $b_{\Phi}< \8.$

\end{enumerate}
See \cite{M1985, M1989} for more information on Orlicz functions and Orlicz spaces.

For an Orlicz function $\Phi,$ the noncommutative Orlicz space
$L_{\Phi}(\M)$ is defined as the space of all measurable operators $x$ with
respect to $(\M,\tau)$ such that
$$\tau \Big ( \Phi \Big ( \frac{|x|}{c} \Big ) \Big )<\infty$$
for some $c>0.$ The space $L_{\Phi}(\M),$ equipped with the norm
\be
\|x\|_{\Phi}= \inf \big \{c>0: \;\tau \big ( \Phi({|x|}/{c}) \big )<1 \big \},
\ee
is a Banach space. If $\Phi(t)=t^p$ with $1 \leq p<\infty$ then
$L_\Phi(\M)= L_p(\M).$ Noncommutative Orlicz spaces are symmetric spaces of measurable operators as defined in
\cite{DDP1989, Xu1991}.


\section{Noncommutative weak Orlicz spaces}\label{NcWeakOrliczSpa}

In the sequel, unless otherwise specified, we always denote by $\Phi$ an Orlicz function.
Motivated by \eqref{eq:WeakLpNormMu}, we give the following definition

\begin{definition}\label{df:WeakOrliczSpa}
For an Orlicz function $\Phi,$ define
\be
L_\Phi^w (\mathcal{M}) = \big \{ x \in L_{0}(\mathcal{M}):\; \exists\; c>0 \;\text{such that}\; \sup_{t > 0}
t \Phi( \mu_t (x)/c ) < \8 \big \},
\ee
equipped with
\be
\|x\|_{L_\Phi^w} = \inf \big \{ c>0:\; t \Phi( \mu_t (x)/c ) \leq 1, \forall t>0 \big \}.
\ee
$L_\Phi^w (\mathcal{M})$ is said to be a noncommutative weak Orlicz space.
\end{definition}

\begin{remark}\label{rk:wLp}
\begin{enumerate}[\rm (1)]

\item It is easy to check that
\be
\|x\|_{L_\Phi^w} = \inf \Big \{ c>0:\; \frac{1}{\Phi^{-1}(\frac{1}{t})} \mu_t (x)/c  \leq 1, \forall t>0 \Big \}.
\ee

\item For $0 < p < \8,$ if $\Phi(t)=t^p$ then $L_\Phi^{w}(\mathcal{M})$ is the noncommutative weak $L_p$-space
as shown in \eqref{eq:WeakLpNormMu}.

\item Note that the noncommutative Orlicz space $L_{\Phi} (\M)$ has the following description:
\be
L_{\Phi} (\M) = \Big \{ x \in L_{0}(\mathcal{M}):\; \exists
\:c>0, \; \int^{\8}_0 \Big [  t \Phi \Big ( \frac{\mu_t (x)}{c} \Big )\Big ]\frac{d t}{t}  \leq 1 \Big \}
\ee
with the norm
\be
\|x\|_{L_\Phi} = \inf \Big \{ c>0:\; \int^{\8}_0 \Big [  t \Phi \Big ( \frac{\mu_t (x)}{c} \Big ) \Big ] \frac{d t}{t} \leq 1 \Big \}.
\ee
This shows that $L_\Phi^w (\mathcal{M})$ has a close connection with $L_{\Phi} (\M).$
\end{enumerate}

\end{remark}

We have the following useful characterization of $L^w_{\Phi} (\M).$

\begin{proposition}\label{prop:NcWeakOrliczcharac}
Let $\Phi$ be an Orlicz function. For any $c>0$ we have
\beq\label{eq:DistrSingularEquiv}
\sup_{t >0} t \Phi (\mu_t (x)/c) =\sup_{s > 0} \lambda_s (x) \Phi (s /c),\;\forall x \in L_0 (\M).
\eeq
Consequently,
\be
L_\Phi^w (\mathcal{M}) = \big \{ x \in L_{0}(\mathcal{M}):\; \exists\; c>0 \;\text{such that}\; \sup_{s > 0}
 \lambda_s (x) \Phi( s /c ) < \8 \big \},
\ee
and
\be
\|x\|_{L_\Phi^w} = \inf \big \{ c>0:\; \lambda_s (x) \Phi( s /c ) \leq 1, \forall s > 0 \big \}.
\ee
\end{proposition}

\begin{proof}
Since $\lambda_s (x) = \lambda_{\mu(x)} (s),$ where $\lambda_{\mu(x)}$ is the distribution function of the function
$t \to \mu_t (x)$ with respect the Lebesgue measure in $[0, \8 ),$ it reduces to prove that
\beq\label{eq:DistrSingularFunctEquiv}
\sup_{t >0} t \Phi (f^* (t)/c) =\sup_{s > 0} \lambda_f (s) \Phi (s /c),
\eeq
for any nonnegative measurable function $f$ on $(0,\8),$ where $\lambda_f$ is the distribution function of
 $f$ with respect to the Lebesgue measure on $[0, \8)$ and $f^*$ is the rearrangement function of $f$ defined by
\be
f^* (t) = \inf \{s>0:\; \lambda_f (s) \le t \}.
\ee
To this end, we consider a simple function $f = \sum_k a_k \chi_{A_k},$ where $a_k > 0$ and $A_k$ are
measurable subsets of $[0,\8)$ such that $|A_k| < \8$ and $A_k \cap A_j = \emptyset$ whenever $k \not= j.$
An immediate computation yields \eqref{eq:DistrSingularFunctEquiv} holds for such a function.
Since a nonnegative measurable function can be approximated almost everywhere by a sequence of
nonnegative simple functions from below, a standard argument concludes \eqref{eq:DistrSingularFunctEquiv}
 for any nonnegative measurable function.
\end{proof}

We collect some basic properties of noncommutative Orlicz spaces as follows.

\begin{proposition}\label{prop:ineq}
Let $\Phi$ be an Orlicz function.

\begin{enumerate}[\rm (1)]

\item If $\|x\|_{L_\Phi^{w}} > 0$ then
\be
\sup_{t>0} t \Phi \big ( \mu_t (x)/\|x\|_{L_\Phi^{w}} \big ) \leq 1\; \text{and}\; \sup_{s >0}
\lambda_s (x) \Phi \big ( s / \|x\|_{L_\Phi^{w}} \big ) \leq 1.
\ee

\item $\|\cdot\|_{L_\Phi^{w}}$ is a quasi-norm on $L_\Phi^{w}(\mathcal{M}).$ In particular,
\beq\label{eq:QuasiNormInequa}
\| x + y \|_{L_\Phi^{w}} \le 2 ( \|x\|_{L_\Phi^{w}} + \|y\|_{L_\Phi^{w}}),\; \forall x, y \in L_\Phi^{w}(\mathcal{M}).
\eeq

\item If $\|x\|_{L_\Phi^{w}}\leq 1,$ then
\be
\sup_{t>0} t \Phi(\mu_t (x)) \leq \|x\|_{L_\Phi^{w}}\; \text{and}\; \sup_{s >0} \lambda_s (x) \Phi( s )
\leq \|x\|_{L_\Phi^{w}}.
\ee

\item $\| x\|_{L^w_{\Phi}} \le \| x \|_{L_{\Phi}}$ for any $x \in L_\Phi(\mathcal{M}).$
Consequently, $L_\Phi(\mathcal{M})\subset L_\Phi^{w}(\mathcal{M}).$

\end{enumerate}

\end{proposition}

\begin{proof} (1)\; By the definition of $\|x\|_{L_\Phi^{w}},$
there is a sequence $\{c_k\}\subset \mathbb{R}^{+}$ such that $c_k\downarrow
\|x\|_{L_\Phi^{w}}$ and $t \Phi \big ( \mu_t (x)/ c_k \big ) \leq 1$ for all $t>0.$
Since $\Phi$ is continuous, taking $k \rightarrow \infty$ we obtain the first inequality.
The second inequality follows from \eqref{eq:DistrSingularEquiv} and the first one.

(2)\;If $\|x\|_{L_\Phi^{w}}=0,$ then there is a sequence $\{c_k\} \subset \mathbb{R}^{+}$ such that
 $c_k\downarrow 0$ and $t \Phi(\mu_t (x)/c_k) \leq 1, \forall t>0.$ Since $\Phi (t) \to \8$ as $t \to \8,$
  it is concluded that $\mu_t (x)=0,\;\forall t>0,$ which implies $x=0$ from the fact that
  $\lim_{t \to 0^+} \mu_t (x) = \|x\|.$

It is clear that $\|\alpha x\|_{L_\Phi^{w}}=|\alpha|\|x\|_{L_\Phi^{w}}.$ To prove the
generalized triangle inequality, we let $x,y\in L_\Phi^{w}(\mathcal{M})$ and
$\|x\|_{L_\Phi^{w}}=a,\;\|y\|_{L_\Phi^{w}}=b$ with $a, b>0.$ By (1), we have
\be\begin{split}
t \Phi \Big ( \frac{\mu_t (x + y)}{2(a+b)} \Big )& \leq
t \Phi \Big ( \frac{\mu_{t/2} (x) + \mu_{t/2} (y)}{2(a+b)} \Big )\\
& \le \frac{t}{2} \Phi \Big ( \frac{\mu_{t/2} (x)}{a+b} \Big ) + \frac{t}{2} \Phi \Big ( \frac{\mu_{t/2} (y)}{a+b} \Big )\\
& \leq \frac{a}{a + b} \frac{t}{2}\Phi \Big ( \frac{\mu_{t/2} (x)}{a} \Big ) + \frac{b}{a + b} \frac{t}{2}\Phi \Big
( \frac{\mu_{t/2} (y)}{b} \Big ) \leq 1.
\end{split}\ee
Hence, $ \|x+y\|_{L_\Phi^{w}}\leq 2(a+b)= 2 ( \|x\|_{L_\Phi^{w}}+\|y\|_{L_\Phi^{w}}).$

(3)\; If $\|x\|_{L_\Phi^{w}} =0,$ by (2) the inequality holds. Suppose $\|x\|_{L_\Phi^{w}}=a \leq 1$ and $a \not=0.$
By (1) we have that $t \Phi(\mu_t (x)/a) \leq 1, \forall t>0.$ From the convexity of $\Phi$ and the fact $\Phi(0)=0,$
we have $\Phi(at) \leq a \Phi(t), \forall t>0,$ which implies that
\be
\frac{t}{a}\Phi(\mu_t (x)) \leq t \Phi \big ( \mu_t (x)/a \big ) \leq 1, \;\forall t>0.
\ee
This gives the first inequality. The second inequality follows from \eqref{eq:DistrSingularEquiv} and the first one.

(4)\; Let $x \in L_{\Phi}(\mathcal{M}),\;x \neq 0.$ Then, for any $t>0,$
\be\begin{split}
t \Phi \Big ( \frac{\mu_t (x)}{\|x\|_{L_{\Phi}}} \Big ) \le \int^{t}_0
\Phi \Big ( \frac{\mu_s (x)}{\|x\|_{L_{\Phi}}} \Big ) d s \le \int^{\8}_0
\Phi \Big ( \frac{\mu_s (x)}{\|x\|_{L_{\Phi}}} \Big ) d s \leq 1.
\end{split}\ee
Hence, $\|x\|_{L_\Phi^{w}}\leq \|x\|_{L_{\Phi}}$ and $L_{\Phi}(\mathcal{M})\subset L_\Phi^{w}(\mathcal{M}).$
\end{proof}

Recall that for measurable operators $x_n, x$ with respect to $(\M,\tau),$ $x_n$ converges to $x$ in measure
if and only if $\lim_n \mu_t (x_n -x) =0$ for all $t >0.$ Then, we have

\begin{proposition}\label{prop:conv}
Let $\Phi$ be an Orlicz function.

\begin{enumerate}[\rm (1)]

\item If $\|x_n-x\|_{L_\Phi^{w}}\rightarrow 0,$ then $x_n\rightarrow x$ in measure.

\item $L_\Phi^{w}(\mathcal{M})$ is a quasi-Banach space.

\end{enumerate}

\end{proposition}

\begin{proof} (1)\; Suppose $\|x_n-x\|_{L_\Phi^{w}}\rightarrow 0.$ Then there is a sequence $(c_n)$ of positive numbers
with $\lim_n c_n =0$ such that
\be
t \Phi \Big ( \frac{\mu_t (x_n -x)}{c_n} \Big ) \le 1,\; \forall t >0.
\ee
for all $n.$ Since $\Phi (t) \to \8$ as $t \to \8,$ it is concluded that $\lim_n \mu_t (x_n -x) =0$ for any $t>0.$
Hence, $x_n\rightarrow x$ in measure.

(2)\; By Proposition \ref{prop:ineq} (2), it suffices to prove that $L_\Phi^{w}(\mathcal{M})$ is complete.
Suppose $x_n\in L_\Phi^{w}(\mathcal{M})$ such that $\lim_{m,n \rightarrow\infty} \| x_n - x_m \|_{L_\Phi^{w}}= 0.$
Then, for any $1>\varepsilon >0$ there is an $n_0$ such that $\|x_n-x_m\|_{L_\Phi^{w}}< \varepsilon$ for all
$n,m\geq n_0.$ Since $L_0 (\M)$ is complete in the topology of the convergence in measure, by (1) there exists
$x \in L_0 (\M)$ such that
$$
\lim_{n \rightarrow\infty}\mu_{t}(x_n-x)=0, \forall t>0.
$$
Clearly,
$$
x_{n}-x_{m} \rightarrow x_{n}-x \quad\mbox{in measure}$$
as $m\rightarrow\infty.$ By Proposition \ref{prop:ineq} (3), for any $n \ge n_0$ we have
\be\begin{split}
t \Phi \Big ( \frac{\mu_t (x_n-x)}{\varepsilon} \Big )\le \lim_{m \to \infty} t \Phi \Big ( \frac{\mu_t (x_n-x_m)}
{\varepsilon} \Big ) \leq \liminf_{m\to\infty} \Big \| \frac{x_n-x_m}{\varepsilon} \Big \|_{L_\Phi^{w}}\leq 1,
\end{split}\ee
for any $t > 0.$ This yields $\|x_{n}-x\|_{L_\Phi^{w}}< \varepsilon$ and so
$\lim_{n\rightarrow\infty}\|x_{n}-x\|_{L_\Phi^{w}}=0.$ Also, by \eqref{eq:QuasiNormInequa} we obtain that
 $x\in L_\Phi^{w}(\mathcal{M}).$ Hence, $L_\Phi^{w}(\mathcal{M})$ is complete.
\end{proof}

\begin{remark}\label{rk:NcWeakOrlicz}
Clearly, $L^w_{\Phi} (\M)$ is rearrangement invariant. Then, by Proposition \ref{prop:conv} (2) we have that
$L^w_{\Phi} (\M)$ is a symmetric quasi-Banach space of measurable operators as defined in \cite{Xu1991}.
\end{remark}










The following are two examples for illustrating noncommutative weak Orlicz spaces.

\begin{example}\label{ex:pleq}
Let $\Phi (t) = t^a \ln (1 + t^b)$ with $a > 1$ and $b >0.$ It is easy to check that $\Phi$ is an Orlicz function and
$p_{\Phi} = a$ and $q_{\Phi} = a + b.$ Then, $L^w_{\Phi}$ can not be coincide with any $L^w_p.$
\end{example}

\begin{example}\label{ex:p=q=2}
Let $\Phi (t) = t^p (1 + c \sin(p \ln t))$ with $p > 1/(1-2c)$ and $0< c <1/2.$ Then, $\Phi$ is an
Orlicz function and $p_{\Phi} = q_{\Phi} = p.$ It is clear that $\Phi$ is equivalent to $t^p$ and hence
 $L^w_{\Phi} = L^w_p.$
\end{example}


Let $a=(a_{n})$ be a finite sequence in $L^w_{\Phi}({\mathcal{M}}),$
we define
\be
\|a\|_{L^w_{\Phi}({\mathcal{M}},\el_{C}^{2})}= \Big \| \Big ( \sum_n
|a_{n}|^{2} \Big )^{1/2} \Big \|_{L^w_{\Phi}}\;
\text{and}\;
\|a\|_{L^w_{\Phi}({\mathcal{M}},\el_{R}^{2})} = \Big \| \Big ( \sum_{n\geq 0
} |a_{n}^{*}|^{2} \Big )^{1/2} \Big \|_{L^w_{\Phi}},
\ee
respectively. Then, we have

\begin{proposition}\label{prop:WeakOrliczNorm}
$\|\cdot \|_{L^w_{\Phi}({\mathcal{M}},\el_{C}^{2})}$ and $\|\cdot\|_{L^w_{\Phi}({\mathcal{M}},\el_{R}^{2})}$
 are two quasi-norms on the family of all finite sequences in
$L^w_{\Phi}({\mathcal{M}}).$
\end{proposition}

\begin{proof}
To see this, let us consider the von Neumann algebra tensor product ${\mathcal{M}}\bar{\otimes}{ \mathcal{B}}(\el^{2})$
 with the product trace $\tau \bar{\otimes} \mathrm{tr},$ where
${\mathcal{B}}(\el^{2}) $ is the algebra of all bounded operators on
$\el^{2}$ with the usual trace $\mathrm{tr}.$ $\tau\otimes \mathrm{tr}$ is a
semifinite normal faithful trace. The associated noncommutative weak
Orlicz space is denoted by $L^w_{\Phi}({\mathcal{M}}\bar{\otimes}
{\mathcal{B}}(\el^{2})).$ Now, any finite sequence $a=(a_{n})_{n\geq
0}$ in $L^w_{\Phi}({\mathcal{M}})$ can be regarded as an element in
$L^w_{\Phi}({\mathcal{M}}\bar{\otimes} {\mathcal{B}}(\el^{2}))$ via the
following map
\be
a\longmapsto T(a)= \left(
\begin{matrix}
a_{0} & 0 & \ldots \\
a_{1} & 0 & \ldots \\
\vdots & \vdots & \ddots
\end {matrix}
\right ),
\ee
that is, the matrix of $T(a)$ has all vanishing entries except those
in the first column which are the ${a_{n}}$'s. Such a matrix is
called a column matrix, and the closure in
$L^w_{\Phi}({\mathcal{M}}\bar{\otimes} {\mathcal{B}}(\el^{2}))$ of all column
matrices is called the column subspace of
$L^w_{\Phi}({\mathcal{M}}\bar{\otimes} {\mathcal{B}}(\el^{2})).$ Since
\be
\|a\|_{L^w_{\Phi}({\mathcal{M}},\el_{C}^{2})}=\||T(a)|\|_{L^w_{\Phi}({\mathcal{M}}\bar{\otimes}
{\mathcal{B}}(\el^{2}))}= \|T(a)\|_{L^w_{\Phi}({\mathcal{M}}\bar{\otimes} {\mathcal{B}}(\el^{2}))},
\ee
then $\|.\|_{L^w_{\Phi}({\mathcal{M}},\el_{C}^{2})}$ defines a quasi-norm
on the family of all finite sequences of $L^w_{\Phi}({\mathcal{M}}).$ Similarly, we can show that
 $\|.\|_{L^w_{\Phi}({\mathcal{M}},\el_{R}^{2})}$ defines a quasi-norm
on the family of all finite sequences of $L^w_{\Phi}({\mathcal{M}}).$
\end{proof}

We define $L^w_{\Phi}({\mathcal{M}},\el_{C}^{2})$ (resp. $L^w_{\Phi}({\mathcal{M}}, \el_{R}^{2})$) to be the completion of all finite sequences in $L^w_{\Phi} (\M)$ under the norm $\|\cdot \|_{L^w_{\Phi}({\mathcal{M}},\el_{C}^{2})}$ (resp. $\|\cdot\|_{L^w_{\Phi}({\mathcal{M}},\el_{R}^{2})}$). It is clear that a sequence
$a=(a_{n})_{n\geq 0}$ in $L^w_{\Phi}({\mathcal{M}})$ belongs to
$L^w_{\Phi}({\mathcal{M}}, \el_{C}^{2})$ (resp. $L^w_{\Phi}({\mathcal{M}}, \el_{R}^{2})$) if and only if
\be\begin{split}
\| a \|_{L^w_{\Phi}({\mathcal{M}},\el_{C}^{2})} & : = \sup_{n\geq 0} \Big \| \Big ( \sum_{k= 0
}^{n}|a_{k}|^{2} \Big )^{\frac{1}{2}} \Big \|_{\Phi}<\infty\;\\
\Big ( \mathrm{resp.}\;\|a\|_{L^w_{\Phi}({\mathcal{M}},\el_{R}^{2})} & := \sup_{n\geq 0} \Big \| \Big ( \sum_{k= 0
}^{n}|a^*_{k}|^{2} \Big )^{\frac{1}{2}} \Big \|_{\Phi}<\infty \Big ).
\end{split}\ee
$L^w_{\Phi}({\mathcal{M}},\el_{C}^{2})$ and $L^w_{\Phi}({\mathcal{M}}, \el_{R}^{2})$ are evidently quasi-Banach spaces, but we will see in Sect.\ref{Interp} that they can be renormed as Banach spaces provided $\Phi$ satisfies a mild condition.

\section{Interpolation}\label{Interp}

The main result of this section is a Marcinkiewicz type interpolation theorem for noncommutative weak Orlicz spaces.
We first introduce the following definition.

\begin{definition}\label{df:sublinearOper}
Let $\M$ (resp. $\mathcal{N}$) be a von Neumann algebra with a normal semifinite faithful trace $\tau$ (resp. $\nu$). A map
$T:L_{0}(\mathcal{M})\rightarrow L_{0}(\mathcal{N})$ is said to be quasilinear if
\begin{enumerate}[\rm (i)]

\item $|T(\alpha x)| \leq |\alpha||Tx|$ for all $x \in L_{0}(\mathcal{M})$ and $\alpha \in \mathbb{C};$ and

\item there is a constant $K>0$ so that for arbitrary operators $x,y\in L_{0}(\mathcal{M}),$
there exist two partial isometrics $u,v\in \mathcal{N}$ such that
\be
|T(x+y)|\leq K \big (u^{*}|Tx|u+v^{*}|T y| v \big ).
\ee

\end{enumerate}
In addition, if $K=1$ we call $T$ a sublinear operator.
\end{definition}

This definition of sublinear operators in the noncommutative setting is due to Q.Xu and first appeared in
Ying Hu's thesis \cite{Hu2007} (see also \cite{Hu2009}). Recall that for any $x,y \in L_0 (\N)$ there exist two partial isometrics $u,v\in \mathcal{N}$ such that
\beq\label{eq:OperatorModuleInequa}
| x + y| \le u^* |x| u + v^* |y| v,
\eeq
(see \cite{AAP}) and then a linear operator is sublinear. We recall that a quasilinear operator $T:L_{0}(\mathcal{M})\rightarrow L_{0}(\mathcal{N})$ is of weak type $(p,q)$ with $0<p \le q \le \infty,$ if
$$
\|Tx\|_{L^w_q}\leq C\|x\|_{L_p},~~ \forall x\in L_p(\mathcal{M}).
$$

The classical Marcinkiewicz interpolation theorem has been extended to include Orlicz spaces as interpolation
classes by A.Zygmund, A.P.Calder\'{o}n, S.Koizumi, I.B.Simonenko, W.Riordan, H.P.Heinig and A.Torchinsky
(for references see \cite{M1989}). The following result is a noncommutative analogue of the Marcinkiewicz type interpolation theorem for weak Orlicz spaces.

\begin{theorem}\label{th:inter}
Let $\M$ (resp. $\mathcal{N}$) be a von Neumann algebra with a normal semifinite faithful trace $\tau$ (resp. $\nu$). Suppose $0<p_{0}<p_{1}\leq \infty.$ Let $T :L_{0}(\mathcal{M})\rightarrow L_{0}(\mathcal{N})$ be a quasilinear operator and simultaneously of weak type $(p_i, p_i)$ for $i=0$ and
$i=1.$ If $\Phi$ is an Orlicz function with $p_{0}< a_{\Phi} \le b_{\Phi}< p_{1},$ then there exists a constant $C>0$ such that
\beq\label{eq:WeakInterInequa}
\sup_{t>0}t \Phi \big [ \mu_t (Tx) \big ] \leq C \sup_{t>0} t \Phi \big [ \mu_t (x) \big ]
\eeq
for all $x \in L^w_{\Phi}(\mathcal{M}).$ Consequently,
\beq\label{eq:WeakNormInterInequa}
\| Tx \|_{L^w_{\Phi} (\N)} \lesssim \| x \|_{L^w_{\Phi} (\M)},\quad \forall x \in L^w_{\Phi}(\mathcal{M}).
\eeq
\end{theorem}

\begin{proof}
We chose $\theta_1, \theta_2, r_0, r_1$ such that
\be
p_{0} < r_0 < a_{\Phi} \leq b_{\Phi} < r_1 < p_1
\ee
and
\be
0 < \theta_1, \theta_2< 1,\; \frac{1}{r_{k}} = \frac{(1 - \theta_{k})}{p_{0}} + \frac{\theta_{k}}{p_{1}},\; k = 0, 1.
\ee
Then, by the real interpolation of noncommutative $L_p$ spaces (cf., Corollary 1.6.11 of \cite{Xu2008}), we have
\be
(L_{p_{0}}(\mathcal{M}), L_{p_{1}}(\mathcal{M}))_{\theta_{k},q}= L_{r_{k},q}(\mathcal{M}), k=0,1,
\ee
with equivalent quasi-norms. Since $T$ is simultaneously of weak type $(p_i, p_i)$ for $i=0$ and
$i=1,$ we obtain that
\begin{equation}\label{eq:weakr0}
\|Tx\|_{L^w_{r_0}}\leq A_0\|x\|_{L^w_{r_0}},\quad \forall x \in L^w_{p_0}(\mathcal{M}),
\end{equation}
and
\begin{equation}\label{eq:weakr1}
\|Tx\|_{L^w_{r_1}}\leq A_1\|x\|_{L^w_{r_1}}, \quad \forall x \in L^w_{p_1}(\mathcal{M}),
\end{equation}
where $A_0,A_1$ are both constants which depend only on $p_0, p_1,$ and the weak type $(p_i, p_i)$ norms of $T$ for $i=0$ and
$i=1.$

Now, take $x \in L^w_{\Phi}(\mathcal{M}).$ For any $\alpha>0$ let $x=x_0^\alpha+x_1^\alpha,$
where $x_0^\alpha= x e_{(\alpha,\infty)}(|x|).$ Since $t^{-r_0}\Phi(t)$ is an
increasing function in $(0, \8),$ by Proposition \ref{prop:ineq} (1) and \eqref{eq:weakr0} we have
\be\begin{split}
\lambda_{\alpha}( Tx_0^\alpha ) \leq & \alpha^{- r_0}\|Tx_0^\alpha\|^{r_0}_{L^w_{r_0}}\\
\leq & \alpha^{- r_0} A_0^{r_0}\|x_0^\alpha\|^{r_0}_{L^w_{r_0}}\\
= & \alpha^{- r_0} A_0^{r_0} \sup_{t>0}t^{r_0}\lambda_{t}( x_0^\alpha )\\
\leq & A_0^{r_0} \sup_{t>\alpha} \Big ( \frac{t}{\alpha} \Big )^{r_0}\lambda_{t}( x )\\
\leq & A_0^{r_0} \sup_{t>\alpha}\frac{\Phi(t)}{\Phi(\alpha)}\lambda_{t}( x ) \\
\le & \frac{A_0^{r_0}}{\Phi (\alpha)} \sup_{t>0} \Phi(t) \lambda_{t}( x ).
\end{split}\ee
Also, since $t^{-r_1}\Phi(t)$ is a decreasing function in $(0, \8),$ by Proposition \ref{prop:ineq} (1) and
\eqref{eq:weakr1} we obtain similarly
\be
\lambda_{\alpha}(Tx_1^\alpha) \leq \frac{A_1^{r_1}}{\Phi(\alpha)} \sup_{t>0}\Phi(t)\lambda_{t}(x).
\ee

On the other hand, by the sublinearity of $T$ and the basic properties of the distribution function $\lambda (|x|),$ such as $\lambda (a^* a) = \lambda (a a^*)$ and $\lambda_{\alpha + \beta} (x + y) \le \lambda_{\alpha} (x) + \lambda_{\beta} (y)$ for any $x , y \ge 0$, we have that
\begin{equation}\label{eq:lamb1}
\begin{split}
\lambda_{2 K \alpha} ( T x ) & \le \nu \big ( E_{(2 K \alpha, \8)} \big [ K(u^* | T x^{\alpha}_0| u + v^* |T x^{\alpha}_1| v ) \big ] \big )\\
& \le \lambda_{\alpha} (u^* | T x^{\alpha}_0 | u) + \lambda_{\alpha} (v^* |T x^{\alpha}_1| v)\\
& \leq \lambda_{\alpha}(|Tx_0^\alpha|)+\lambda_{\alpha}(|Tx_1^\alpha|),
\end{split}\end{equation}
where the first and third inequalities use the fact that $0 \le a \le b$ implies $E_{(\alpha, \8)} (a)$ is equivalent to a subprojection of $E_{(\alpha, \8)} (b)$ (e.g., \cite{FK}). Then, by \eqref{eq:lamb1} we have
\be\begin{split}
\lambda_{2K\alpha} (Tx)\leq & \frac{A_0^{r_0}}{\Phi(2 K \alpha)} \sup_{t>0}\Phi(t)\lambda_{t}(x) +
\frac{A_1^{r_1}}{\Phi(2 K \alpha)}\sup_{t>0}\Phi(t)\lambda_{t}(x)\\
= & \frac{C}{\Phi(2 K \alpha)} \sup_{t>0}\Phi(t)\lambda_{t}(x).
\end{split}\ee
By Proposition \ref{prop:NcWeakOrliczcharac} we obtain the desired inequality \eqref{eq:WeakInterInequa}.
\end{proof}

\begin{remark}\label{re:Inter}

We set
\be
L_p (\N)_{\mathrm{Her}} = \{x \in L_p (\N):\; x^* =x \}.
\ee
If $T$ is simultaneously of weak types $L_{p_i} (\M)_{\mathrm{Her}} \to L_{p_i} (\N)_{\mathrm{Her}} $ for $i=0$ and $i=1,$ then the conclusion of Theorem \ref{th:inter} holds for any hermitian operator $x \in L_{\Phi} (\M).$ The proof is the same as above and omitted.

\end{remark}

We have the following corollaries.

\begin{corollary}\label{cor:interstrong}
Let $\M$ (resp. $\mathcal{N}$) be a von Neumann algebra with a normal semifinite faithful trace $\tau$ (resp. $\nu$). Suppose $0<p_{0}<p_{1}\leq \infty.$
Let $T: L_0 (\M) \mapsto L_0 (\N)$ be a quasilinear operator and simultaneously of strong type $(p_i, p_i)$ for $i=0$ and
$i=1,$ i.e.,
\be
\|Tx\|_{L_{p_0}} \lesssim \|x\|_{L_{p_0}},\quad \forall x \in L_{p_0}(\mathcal{M}),
\ee
and
\be
\|Tx\|_{L_{p_1}} \lesssim \|x\|_{L_{p_1}},\quad \forall x \in L_{p_1}(\mathcal{M}).
\ee
Let $\Phi$ be an Orlicz function with $p_0 < a_{\Phi} \le b_{\Phi} < p_1.$ Then,
the conclusion of Theorem \ref{th:inter} holds.
\end{corollary}

\begin{proof}
If $T$ is of strong type $(p,p),$ by the Kolmogorov inequality \eqref{eq:kol} we immediately conclude that $T$ is of weak type $(p,p).$ An appeal to Theorem \ref{th:inter} yields the result.
\end{proof}

\begin{corollary}\label{cor:NcW-norm}
Let $\Phi$ be an Orlicz function with $1< a_{\Phi} \le b_{\Phi}<\8.$ Then
\beq\label{eq:WeakOrliczRenorm}
\|x\|_{L_\Phi^w} \approx \inf \Big \{ c>0:\; t \Phi \Big ( \frac{1}{t}\int_{0}^{t}\mu_s (x)d s /c \Big ) \leq 1, \forall t>0 \Big \}.
\eeq
Consequently, $ L_\Phi^w (\mathcal{M}) $ can be renormed as a Banach space.
\end{corollary}

\begin{proof} Since $\mu_t (x)$ is decreasing in $t \in (0, \8),$ we immediately get
\be
\|x\|_{L_\Phi^w} \le \inf \Big \{ c>0:\; t \Phi \Big ( \frac{1}{t}\int_{0}^{t}\mu_s (x)d s /c \Big ) \leq 1, \forall t>0 \Big \}.
\ee
Conversely, let $1< p \le \infty.$ Define $S:\; f(t)\mapsto \frac{1}{t}\int_{0}^{t}|f(s)|ds$ for $f \in L_{p}(0,\infty).$
Then, by the classical Hardy-Littlewood inequality there exists a constant $A_p>0$ such that
$$
\|S f \|_{p}\leq C_p\|f\|_{p},\quad \forall f\in L_{p}(0,\infty).
$$
Consequently,
\be
\|T x \|_p \leq A_p \| x \|_p,\quad \forall x \in L_p (\M),
\ee
where
\be
T x : = \frac{1}{t}\int_{0}^{t}\mu_s (x) d s, \quad x \in L_0 (\M).
\ee
Since $T$ is sublinear, by Corollary \ref{cor:interstrong} we obtain the reverse inequality and hence \eqref{eq:WeakOrliczRenorm} holds.
\end{proof}

\begin{corollary}\label{cor:BoydIndex}
Let $\Phi$ be an Orlicz function with $1<a_{\Phi}\le b_{\Phi}<\infty.$ Let $p_\Phi^w $ and $q_\Phi^w $ be respectively the lower and upper Boyd indices of $L_\Phi^w (\mathcal{M}).$ Then,
\begin{equation}\label{eq:Boyd}
a_{\Phi}\le p_\Phi^w \le q_\Phi^w \le b_{\Phi}.
\end{equation}
\end{corollary}

\begin{proof}
Let $1\le p<a_{\Phi}\le b_{\Phi}<q<\infty.$ Suppose $T$ is a linear operator defined on $L_{p,1} [0,\infty )+L_{q,1} [0,\infty ),$
which is simultaneously of weak type $(p,p)$ and weak type $(q,q)$ in the sense of \cite{LT}. Take $p_{0}, q_{0}$ such that $p<p_{0}<a_{\Phi}\le
b_{\Phi}< q_{0}<q,$ Then by Theorem 2.b.11 in \cite{LT}, we have that $T$ is simultaneously of strong type $(p_{0},p_{0})$ and  strong type $(q_{0},q_{0}).$
Using Corollary \ref{cor:interstrong}, we get $T$ maps $L_\Phi^w (\mathcal{M})$ into itself. Then, by Theorem 2.b.13 in \cite{LT} we conclude that $ p<p_\Phi^w \le
q_\Phi^w <q.$ This completes the proof.
\end{proof}

\section{Martingale inequalities}\label{MartingaleInequa}

In this section, we will prove the weak type $\Phi$-moment versions of martingale transformations, Stein's inequalities, Khintchine's inequalities for Rademacher's random variables, and Burkholder-Gundy martingale inequalities in the noncommutative setting. We mainly follows the arguments in \cite{BC} using Theorem \ref{re:Inter} and Corollary \ref{cor:interstrong}.

In the sequel, without otherwise specified, we always denote by ${\mathcal{M}}$ a finite von Neumann algebra
with a normalized normal faithful trace $\tau.$ Let $({\mathcal{M}}_{n})_{n\geq 0}$ be an increasing sequence of von
Neumann subalgebras of ${\mathcal{M}}$ such that $\cup_{n\geq 0} {\mathcal{M}}_{n}$ generates ${\mathcal{M}}$ (in the $w^{*}$-topology). $({\mathcal{M}}_{n})_{n\geq 0}$ is called a filtration of ${\mathcal{M}}.$ The restriction of $\tau$ to ${\mathcal{M}}_{n}$ is still denoted by $\tau.$ Let ${\mathcal{E}}_{n}={\mathcal{E}}(.|{\mathcal{M}}_{n})$ be the conditional expectation of ${\mathcal{M}}$ with respect to
${\mathcal{M}}_{n}.$

A non-commutative $L^w_{\Phi}$-martingale with respect to $({\mathcal{M}}_{n})_{n\geq 0}$ is a sequence
 $x=(x_{n})_{n\geq 0}$ such that $x_{n} \in L^w_{\Phi}({\mathcal{M}}_{n})$ and
$${\mathcal{E}}_n(x_{n+1})=x_n$$
for any $n \ge 0.$ Let $\|x\|_{L^w_{\Phi}}=\sup_{n\geq 0}\|x_{n}\|_{L^w_{\Phi}}.$ If $\|x\|_{L^w_{\Phi}}
<\infty,$ then $x$ is said to be a bounded $L^w_{\Phi}$-martingale.

For convenience, we denote the weak type $\Phi$-moment of $x$ by
\be
\| x\|_{\Phi_w (\M)} : = \sup_{t > 0} t \Phi (\mu_t (x)), \quad x \in L_0 (\M).
\ee
We write $\| x\|_{\Phi_w } = \| x\|_{\Phi_w (\M)}$ in short when no confusion occurs.

Let $\alpha=(\alpha_{n})\subset\mathbb{C}$ be a sequence. Recall that a
map $T_{\alpha}$ on the family of martingale difference sequences
defined by $T_{\alpha}(dx)=(\alpha_{n}dx_{n})$ is called
 the martingale transform of symbol $\alpha$. It is clear that $(\alpha_{n}dx_{n})$ is indeed a martingale
 difference sequence. The corresponding martingale is $T_{\alpha}(x)=\sum_n \alpha_{n}dx_{n}.$

\begin{theorem}\label{th:MartingaleTrans}
Let $\alpha=(\alpha_{n})\subset\mathbb{C}$ be a bounded sequence and $T_{\alpha}$ the associated martingale transform. Let $\Phi$ be a Orlicz function such that $1<a_{\Phi} \leq b_{\Phi}<\infty.$ Then, for all bounded $L_{\Phi}^{w}$-martingales
$x=(x_{n}),$ we have
\begin{equation}\label{eq:MartingaleTrans}
\| T_{\alpha} x\|_{\Phi_w} \lesssim \| x\|_{\Phi_w},
\end{equation}
where $\lesssim$ depends only on $\Phi$ and $\sup_n |\alpha_n|.$ Consequently,
\beq\label{eq:MartingaleTransEquiv}
\| x\|_{\Phi_w}  \approx \Big \| \sum \varepsilon_n d x_n \Big \|_{\Phi_w},\quad \forall \varepsilon_n = \pm 1,
\eeq
for any bounded $L_{\Phi}^{w}$-martingales $x=(x_{n}),$ where ``$\approx$" depends only on $\Phi.$
\end{theorem}

\begin{proof}
By the $L_p$-boundedness of martingale transforms (see \cite{PX1997}) and Corollary \ref{cor:interstrong}, we immediately conclude
\eqref{eq:MartingaleTrans} and so \eqref{eq:MartingaleTransEquiv}.
\end{proof}

As in \cite{PX1997}, consider the mapping $T$ defined in $L_p (\M \bar{\otimes} \mathcal{B} (\el^2))$ by
$$
T \left(
\begin{matrix} a_{11} & \ldots & a_{1n} & \ldots \\
a_{21} & \ldots & a_{2n} & \ldots \\
\vdots & \vdots & \vdots & \vdots \\
a_{n1} & \ldots & a_{nn} & \ldots \\
\vdots & \vdots & \vdots & \ddots \end{matrix}\right) = \left(
\begin{matrix}
{\mathcal{E}}_1 (a_{11}) & 0 & 0 & \ldots\\
{\mathcal{E}}_2 (a_{21}) & 0 & 0 & \ldots \\
\vdots & \vdots & \vdots       & \vdots \\
{\mathcal{E}}_{n}(a_{n1}) & 0 & 0 & \ldots\\
 \vdots & \vdots & \vdots & \ddots
\end{matrix}\right).
$$
It is proved in \cite{PX1997} that $T$ is bounded on $L_p (\M \bar{\otimes} \mathcal{B} (\el^2))$ for any $1< p < \8.$ Then, by Corollary \ref{cor:interstrong} we have

\begin{theorem}\label{th:SteinInequC}
Let $\Phi$ be an Orlicz function with $1< a_{\Phi} \le b_{\Phi} < \8.$ Then,
\beq\label{eq:SteinInequC}
\Big \| \Big (\sum_{n} |{\mathcal{E}}_{n} (a_{n}) |^{2} \Big )^{\frac{1}{2}}\Big \|_{\Phi_w} \lesssim \Big \| \Big( \sum_{n}|a_{n}|^{2} \Big )^{\frac{1}{2}} \Big \|_{\Phi_w}
\eeq
for any finite sequence $(a_n)$ in $L^w_{\Phi} (\M).$ Similarly, we have
\beq\label{eq:SteinInequC}
\Big \| \Big (\sum_{n} |{\mathcal{E}}_{n} (a^*_{n}) |^{2} \Big )^{\frac{1}{2}}\Big \|_{\Phi_w} \lesssim \Big \| \Big( \sum_{n}|a_{n}|^{2} \Big )^{\frac{1}{2}} \Big \|_{\Phi_w}
\eeq
for any finite sequence $(a_n)$ in $L^w_{\Phi} (\M).$
\end{theorem}

The following is the weak type $\Phi$-moment version of noncommutative Kintchine's inequalities for Rademacher's sequences.

\begin{theorem}\label{th:KhintInequa}
Let $\Phi$ be an Orlicz function and $\{ \varepsilon_k \}$ a Rademacher's sequence on a probability space $(\Omega, P).$
\begin{enumerate}[\rm (1)]

\item If $1< a_{\Phi} \le b_{\Phi} <2,$ then for any finite sequence $\{ x_k \}$ in $L^w_{\Phi} (\M)$
\begin{equation}\label{eq:Khint1}
\begin{split}
\Big \| & \sum_k \varepsilon_k x_k \Big \|_{\Phi_w(L_{\8}(\Omega) \bar{\otimes} \M)}\\
& \approx \inf \left \{ \Big \| \Big ( \sum_k |y_k|^2 \Big )^{\frac{1}{2}} \Big \|_{\Phi_w (\M)} + \Big \| \Big ( \sum_k |z^*_k|^2 \Big )^{\frac{1}{2}} \Big \|_{\Phi_w (\M)} \right \}
\end{split}\end{equation}
where the infimum runs over all decompositions $x_k = y_k + z_k$ with $y_k, z_k \in L^w_{\Phi} (\M)$ and $``\approx"$ depends only on $\Phi.$

\item If $2< a_{\Phi} \le b_{\Phi} <\8,$ then for any finite sequence $\{ x_k \}$ in $L^w_{\Phi} (\M),$
\begin{equation}\label{eq:Khint2}
\begin{split}
\Big \| & \sum_k \varepsilon_k x_k \Big \|_{\Phi_w(L_{\8}(\Omega) \bar{\otimes} \M)}\\
& \approx \Big \| \Big ( \sum_k |x_k|^2 \Big )^{\frac{1}{2}} \Big \|_{\Phi_w (\M)} + \Big \| \Big ( \sum_k |x^*_k|^2 \Big )^{\frac{1}{2}} \Big \|_{\Phi_w (\M)}
\end{split}\end{equation}
where $``\approx"$ depends only on $\Phi.$
\end{enumerate}
\end{theorem}

\begin{proof}
By the argument in \cite{BC}, we need only to prove the lower estimate of \eqref{eq:Khint1}. By the analogue argument in \cite{LPP}, we are reduced to show for any finite sequence $\{x_k \}$ in $L^w_{\Phi}(\M),$
\begin{equation}\label{eq:Khint1'}
\begin{split}
\inf \Big \{ \Big \| \Big ( & \sum_{k=0}^{n} |y_{k}|^2 \Big )^{ \frac{1}{2}} \Big \|_{\Phi_w} +
\Big \| \Big ( \sum_{k=0}^{n} |z_{k}^{*}|^{2} \Big )^{\frac{1}{2} } \Big \|_{\Phi_w} \Big \}\\
& \lesssim \Big \| \sum_{k=0}^{n}x_{k}z^{3^k} \Big \|_{\Phi_w (L_{\8} (\mathbb{T}) \bar{\otimes} \M)},
\end{split}
\end{equation}
where the infimun runs over all decomposition $x_k = y_k + z_k$ with $y_{k}$ and $z_{k}$ in $L^w_{\Phi}({\M}).$

To this end, we consider $\mathcal{N}=L_{\infty}(\mathbb{T})\bar{\otimes}\mathcal{M}$ equipped with the tensor product trace $\nu=\int\otimes\tau$  and
$\mathcal{A}= \mathcal{H}_{\infty}(\mathbb{T})\overline{\otimes}\mathcal{M}.$ Then, $\mathcal {A}$ is a finite maximal subdiagonal algebras of
$\mathcal {N}$ with respect to $\mathcal{E}=\int \otimes I_{\mathcal{M}}:\mathcal{N}\rightarrow\mathcal{M}$ (e.g., see \cite{PX2003}). Since $L_1 (\N) = L_1 (\mathbb{T}, L_1 (\M))$ we can define Fourier coefficients for any $f \in L_1 (\N)$ by
\be
\hat{f} (n) = \frac{1}{2 \pi} \int_{\mathbb{T}} f (z) \bar{z}^n d m (z),\; \forall n \in \mathbb{Z},
\ee
where $d m$ is the normalized Lebesgue measure on $\mathbb{T}.$ It is easy to check that
\be
\A = \{ f \in \N:\; \hat{f} (n) = 0,\; \forall n < 0 \}.
\ee

For any $n \in \mathbb{Z}$ we define $F_n$ the linear mapping such that $F_n (f) = \hat{f} (n)$ for any $L_1 (\N).$ Then $F_n$ is both a contract from $L_1 (\N)$ into $L_1 (\M)$ and from $\N$ into $\M.$ Hence, for an Orlicz function $\Phi$ with $1 < a_{\Phi} \le b_{\Phi} < \8,$ by Corollary \ref{cor:interstrong} we have
\beq\label{eq:FourierCoefficientsBound}
\| \hat{f} (n) \|_{\Phi_w} \lesssim \| f \|_{\Phi_w},\quad \forall f \in L^w_{\Phi} (\N),
\eeq
for any $n \in \mathbb{Z}.$

\begin{lemma}\label{le:FourierCoefficientsSquare}Let $\Phi$ be an Orlicz function with $1 < a_{\Phi} \le b_{\Phi} < \8.$ For any finite sequence $(f_k)$ in $L^w_{\Phi} (\N)$ and any $n \in \mathbb{Z},$ we have
\be
\Big \| \Big ( \sum_k | \hat{f}_k (n) |^2 \Big )^{\frac{1}{2}} \Big \|_{\Phi_w (\M)} \lesssim \Big \| \Big ( \sum_k | f_k |^2 \Big )^{\frac{1}{2}} \Big \|_{\Phi_w (\N)}.
\ee

\end{lemma}

\begin{proof}
Let $1 \le k \le K.$ Applying \eqref{eq:FourierCoefficientsBound} on $M_K (\M)$ instead of $\M$ with
\be
f = \sum^K_{k=1} E_{k 1} \otimes f_k = \left (
\begin{matrix}
f_1 & 0 & \ldots & 0\\
f_2 & 0 & \ldots & 0\\
\vdots & \vdots & \vdots & \vdots \\
f_K & 0 & \ldots & 0
 \end{matrix} \right )_{K \times K}.
\ee
yields the required inequality.
\end{proof}

For an Orlicz function $\Phi,$ we denote by $\mathcal{H}^w_{\Phi}
(\mathcal{A})$ the completion of $\mathcal{A}$ under the quasinorm
$\|\cdot \|_{L^w_{\Phi} (\N)}.$ If $\Phi$ is an Orlicz function with $1 < a_{\Phi} \le b_{\Phi} < \8,$ then by Corollary \ref{cor:BoydIndex} we have
$L^w_{\Phi} (\N) \subset L_1 (\N)$ and 
\be 
\mathcal{H}^w_{\Phi} (\mathcal{A}) = \Big \{ f \in L^w_{\Phi} (\N):\; \hat{f} (n) =0,\; \forall n <0 \Big \}. 
\ee
In this case,
\beq\label{eq:WeakOrliczHardyHp} 
\mathcal{H}^1 (\mathcal{A}) \cap L^w_{\Phi} (\N) = \mathcal{H}^w_{\Phi} (\mathcal{A}).
\eeq

\begin{lemma}\label{le:RieszDecompNcWeakOrliczHardy}  Let $\Phi$ be an Orlicz function with $1< a_{\Phi}\leq b_{\Phi}<\infty.$ Let $\Phi^{(2)}(t)=\Phi(t^{2}).$ Then, for any $f \in \mathcal{H}^w_{\Phi}({\mathcal{A}})$ and $\varepsilon>0,$ there exist two functions $g,h\in \mathcal{H}^w_{\Phi^{(2)}}({\mathcal{A}})$ such that $f=gh$ with \be
\max \Big \{\| |g|^2 \|_{\Phi_w (\N)},\; \| |h|^2 \|_{\Phi_w (\N)}\Big \} \lesssim \| |f| \|_{\Phi_w (\N)} + \varepsilon.
\ee
\end{lemma}

\begin{proof}
By slightly modifying the proof of Lemma 4.1 in \cite{BC} we can prove this lemma and omit the details.
\end{proof}

\begin{lemma}\label{le:Fourier} Let $\Phi$ be an Orlicz function with $2<a_{\Phi}\leq b_{\Phi}<\infty.$ Let $\{I_n =(\frac{3^n}{2},3^n]:n\in\mathbb{N}\}$
and $\triangle_{n}$ the Fourier multiplier by the indicator function $\chi_{I_{n}},$ i.e. \be \triangle_{n}(f)(z) = \sum_{k \in I_{n}}
\hat{f}(k)z^{k} \ee for any trigonometric polynomial $f$ with coefficients in $ L^w_{\Phi}({\mathcal{M}}).$ Then,
\be
\Big \| \Big ( \sum_{n}\triangle_{n}(f)^{*}\triangle_{n}(f) \Big )^{\frac{1}{2}}  \Big \|_{\Phi_w (\N)} \lesssim \| f \|_{\Phi_w (\N)},
\ee
for any $ f \in \mathcal{H}^w_{\Phi}(\N).$
\end{lemma}

\begin{proof}
The proof can be done as similar to the one of Lemma 4.2 in \cite{BC} by using Corollary \ref{cor:interstrong} and the details are omitted.
\end{proof}

Now, we are ready to prove \eqref{eq:Khint1'}. Indeed, the proof can be obtained by using Lemmas \ref{le:FourierCoefficientsSquare}, \ref{le:RieszDecompNcWeakOrliczHardy} and \ref{le:Fourier} as similar to the one of Theorem 4.1 in \cite{BC}. We omit the details.
\end{proof}

\begin{remark}
\begin{enumerate}[\rm (1)]

\item Note that Khintchine's inequality is valid for $L_1$-norm in both commutative and noncommutative settings (cf., \cite{LPP}). We could conjecture that the right condition in Theorem \ref{th:KhintInequa} (1) should be $b_{\Phi} <2$ without the additional restriction one $1<a_{\Phi}.$ However, our argument seems to be inefficient in this case. We need new ideas to approach it.

\item Evidently, the weak type $\Phi$-moment Khintchine inequalities in Theorem \ref{th:KhintInequa} imply those for $L^w_{\Phi}$ norms, which, by Corollary \ref{cor:NcW-norm}, can be considered as particular cases of more general ones in \cite{LP1992} and then in \cite{LPX, MS}.
\end{enumerate}
\end{remark}

Now, we are in a position to state and prove the weak type $\Phi$-moment version of noncommutative Burkholder-Gundy martingale inequalities.

\begin{theorem}\label{th:BG}
Let $\M$ be a finite von Neumann algebra with a normalized normal faithful trace $\tau$ and $({\mathcal{M}}_{n})_{n\geq 0}$ an increasing filtration
of subalgebras of ${\mathcal{M}}.$ Let $\Phi$ be an Orlicz function and $x = \{x_{n}\}_{n\geq 0}$ a noncommutative $L^w_{\Phi}$-martingale with respect to $({\mathcal{M}}_{n})_{n\geq 0}.$
\begin{enumerate}[\rm (1)]

\item If $1<a_{\Phi} \le b_{\Phi}<2,$ then
\begin {equation}\label{eq:BG1}
\begin{split}
\| x \|_{\Phi_w} \approx \inf \Big \{ \Big \| \Big ( \sum_{n= 0 }^{\infty} |d y_n |^{2} \Big )^{ \frac{1}{2}} \Big \|_{\Phi_w} + \Big \| \Big ( \sum_{n= 0 }^{\infty} |d z_n^{*}|^{2} \Big )^{ \frac{1}{2}} \Big \|_{\Phi_w} \Big \}
\end{split}
\end{equation}
where the infimum runs over all decomposition $x_n = y_n + z_n$ with $\{ d y_n \}$ in $L^w_{\Phi}({\mathcal{M}}, \el^2_C)$ and $\{ d z_n \}$ in
$L^w_{\Phi}({\mathcal{M}}, \el^2_R)$ and ``$\approx$" depends only on $\Phi.$

\item If $2 <a_{\Phi} \le b_{\Phi}<\infty,$ then
\begin {equation}\label{eq:BG2}
\| x \|_{\Phi_w} \approx \Big \| \Big ( \sum_{n= 0 }^{\infty} |d x_n |^{2} \Big )^{ \frac{1}{2}} \Big \|_{\Phi_w} + \Big \| \Big ( \sum_{n= 0 }^{\infty} |d x_n^{*}|^{2} \Big )^{ \frac{1}{2}} \Big \|_{\Phi_w},
\end{equation}
where ``$\approx$" depends only on $\Phi.$


\end{enumerate}
\end{theorem}

\begin{proof}
The proof is similar to the one of Theorem 5.1 in \cite{BC} through using Theorem \ref{th:KhintInequa} and the details are omitted.
\end{proof}

\begin{remark}
All inequalities above are left open for $1< a_{\Phi} \le 2 \le b_{\phi}<\8.$ At the time of this writing, we do not see how to formulate a meaningful statement for this case. However, our argument works well in the commutative case for all cases $1< a_{\phi} \le b_{\Phi} < \8.$
\end{remark}


\end{document}